\documentclass[11pt]{amsart}
\pdfoutput=1
\usepackage{amsmath, amssymb}
\usepackage{amsfonts}
\usepackage{mathrsfs}
\usepackage[arrow,matrix,curve,cmtip,ps]{xy}
\usepackage{graphicx}
\usepackage{amsthm}
\usepackage{float}
\usepackage{amsthm}
\usepackage[utf8]{inputenc}
\usepackage[T1]{fontenc}
\usepackage{mathtools}
\DeclarePairedDelimiter{\ceil}{\lceil}{\rceil}
\allowdisplaybreaks

\newtheorem{theorem}{Theorem}[section]
\newtheorem{lemma}[theorem]{Lemma}
\newtheorem{proposition}[theorem]{Proposition}

\newtheorem{conjecture}[theorem]{Conjecture}

\newtheorem*{theorem*}{Theorem}
\theoremstyle{remark}

\newtheorem{question}{Question}

\numberwithin{equation}{section}

\begin{document}
\title[ Uniform Hyperbolicity of $\mathcal{C}$]{Uniform Hyperbolicity of the Graphs of Curves}

\author{Tarik Aougab}

\address{Department of Mathematics \\ Yale University \\ 10 Hillhouse Avenue, New Haven, CT 06510 \\ USA}
\email{tarik.aougab@yale.edu}

\date{\today}

\keywords{Uniform Hyperbolicity, Curve Complex, Mapping Class Group}

\begin{abstract}

Let $\mathcal{C}(S_{g,p})$ denote the curve complex of the closed orientable surface of genus $g$ with $p$ punctures. Masur-Minksy and subsequently Bowditch showed that $\mathcal{C}(S_{g,p})$ is $\delta$-hyperbolic for some $\delta=\delta(g,p)$. In this paper, we show that there exists some $\delta>0$ independent of $g,p$ such that the curve graph $\mathcal{C}_{1}(S_{g,p})$ is $\delta$-hyperbolic. Furthermore, we use the main tool in the proof of this theorem to show uniform boundedness of two other quantities which a priori grow with $g$ and $p$: the curve complex distance between two vertex cycles of the same train track, and the Lipschitz constants of the map from Teichm\"{u}ller space to $\mathcal{C}(S)$ sending a Riemann surface to the curve(s) of shortest extremal length.

\end{abstract}

\maketitle

\section{Introduction}
 Let $S_{g,p}$ denote the orientable surface of genus $g$ with $p$ punctures. The $\textit{Curve Complex}$ of $S_{g,p}$, denoted $\mathcal{C}(S_{g,p})$, is the simplicial complex whose vertices are in 1-1 correspondence with isotopy classes of non-peripheral simple closed curves on $S_{g,p}$, and such that $k+1$ vertices span a $k$-simplex if and only if the corresponding $k+1$ isotopy classes can be realized disjointly on $S_{g,p}$. $\mathcal{C}(S)$ is made into a metric space by identifying each $k$-simplex with the standard simplex in $\mathbb{R}^{k}$ with unit length edges (see \cite{Mas-Mur} for more details). 

 Masur and Minsky in \cite{Mas-Mur} showed that there is some $\delta=\delta(S_{g,p})$ such that $\mathcal{C}(S_{g,p})$ is $\delta$-hyperbolic, meaning that geodesic triangles in $\mathcal{C}(S)$ are $\delta-thin$: any edge is contained in the $\delta$-neighborhood of the union of the other two edges. Bowditch reproved this result in \cite{Bow}, showing that the hyperbolicity constant $\delta$ grows no faster than logarithmically in $g$ and $p$. Let $\mathcal{C}_{n}(S)$ denote the $n$-skeleton of $\mathcal{C}(S)$. The main result of this paper is that for $\mathcal{C}_{1}(S)$, the \textit{curve graph}, one may take $\delta$ to be independent of $g$ and $p$:

\begin{theorem} There exists $k>0$ so that for any admissible choice of $g,p$, $\mathcal{C}_{1}(S_{g,p})$ is $k$-hyperbolic.  
\end{theorem}

 Note that $\mathcal{C}(S)$ is quasi-isometric to $\mathcal{C}_{1}(S)$, however the quasi-constants depend on the underlying surface. Therefore Theorem $1.1$ does not imply the uniform hyperbolicity of $\mathcal{C}(S)$. We also note that it has recently come to the author's attention that Brian Bowditch has independently obtained the same result \cite{BowII}, as have Clay, Rafi, and Schleimer using different methods \cite{Cl-Ra-Sc}. 

 Let $d_{\mathcal{C}(S_{g,p})}(\cdot,\cdot)$ denote distance in $\mathcal{C}_{1}(S_{g,p})$; when there is no ambiguity, the reference to $S_{g,p}$ will be ommitted in this notation. Let $\xi(S_{g,p})=3g+p-4$; $\xi(S)$ is called the \textit{complexity} of $S$. In the case that $S$ is a disconnected surface, define $\xi(S)$ to be the sum of the complexity of its connected components. If $\xi(S)\leq 0$, $S$ is called $\textit{sporadic}$ and each component of $S$ possesses one of finitely many well understood topological types. In truth, the definition of $\mathcal{C}(S)$ needs to be modified when $S$ is sporadic and connected (see \cite{Mas-Mur} for details), because it is exactly these surfaces for which no two simple closed curves are disjoint.

In what follows, $i(\alpha,\beta)$ is the \textit{geometric intersection number} of $\alpha $ and $\beta$, defined to be the minimum value of $|x\cap y|$, where $x$ and $y$ are representatives of the homotopy classes of $\alpha$, and $\beta$, respectively. The main tool in the proof of Theorem $1.1$ is the following Theorem:

\begin{theorem} For each $\lambda\in (0,1)$, there is some $N=N(\lambda)\in \mathbb{N}$ such that if $\alpha,\beta \in \mathcal{C}_{0}(S_{g,p})$, whenever $\xi(S_{g,p})>N$ and  $d_{\mathcal{C}}(\alpha,\beta)\geq k$, 
\[ i(\alpha,\beta)\geq   \left(\frac{\xi(S)^{\lambda}}{f(\xi(S))}\right)^{k-2}  \]
where $f(\xi)=O(\log_{2}(\xi))$. 

\end{theorem}

In order to emphasize how Theorem $1.2$ is used in the proof of Theorem $1.1$, we define the numbers $q_{k,g,p}$ as follows:
\[ q_{k,g,p}:= \min\left\{ i(\alpha,\beta) | \alpha,\beta\in \mathcal{C}_{0}(S_{g,p}) \hspace{1 mm} \mbox{with} \hspace{1 mm} d_{\mathcal{C}}(\alpha,\beta) =k.\right\} \]
 Irrespective of $\xi$, curve complex distance is bounded above by a logarithmic function of intersection number, and therefore if we fix $g$ and $p$, $q_{k,g,p}$ grows exponentially as a function of $k$. But one may also study $q_{k,g,p}$ as a function of $g$ or $p$:

\begin{question} How does $q_{k,g,p}$ grow as a function of $\xi$? 
\end{question}
The content of Theorem $1.2$ is that $q_{k,g,p}$ eventually grows faster than $\xi^{\lambda(k-2)}$ as a function of $\xi$ for any $\lambda\in (0,1)$. 

We conclude the paper with two further applications of Theorem $1.2$, the first being the resolution of a question posed by Masur and Minsky in \cite{Mas-Mur} regarding vertex cycles of train tracks on surfaces (see section $5$ for relevant definitions). Specifically they ask:

\begin{question} As a consequence of the fact that there are only finitely many train tracks on a surface $S$ up to combinatorial equivalence, there is a bound $B=B(S)$ depending only on the topology of $S$ such that 

\[ d_{\mathcal{C}(S)}(\alpha,\beta)<B \]
if $\alpha,\beta\in \mathcal{C}_{0}(S)$ are two vertex cycles of the same train track $\tau$. Can $B$ be made independent of $S$? That is, is there some $B>0$ such that, for any choice of $S$, the curve complex distance between two vertex cycles of the same train track on $S$ is no more than $B$?

\end{question}

In section 5, we show 

\begin{theorem}
Let $\tau\subset S_{g,p}$ be a train track, and let $\alpha,\beta$ be vertex cycles of $\tau$. Then if $\xi(S_{g,p})$ is sufficiently large, 
\[ d_{\mathcal{C}(S)}(\alpha,\beta)\leq 3. \]
\end{theorem}

 In what follows, let $\mbox{Teich}(S)$ denote the Teichm\"{u}ller space of $S$, the space of marked Riemann surfaces homeomorphic to $S$, modulo conformal equivalence isotopic to the identity. For the remainder of this paper, we will be concerned with the topology on $\mbox{Teich}(S)$ induced by the \textit{Teichm\"{u}ller metric}, denoted by $d_{\mbox{Teich}}(\cdot,\cdot)$. In this metric, the distance between two marked Riemann surfaces $x$ and $y$ is determined by the minimal dilatation constant associated to a quasiconformal map isotopic to the identity between $x$ and $y$ (see \cite{Far-Mar} for more details). 

In the final section, we consider the map $\Phi:\mbox{Teich}(S)\rightarrow \mathcal{C}(S)$ introduced by Masur and Minsky in \cite{Mas-Mur}; here $\Phi(x)$ is the set of isotopy classes of simple closed curves minimizing the extremal length in the conformal class of $x$. Note that $\Phi$ is technically a map into $\mathcal{P}(\mathcal{C}_{0}(S))$ (the power set of $\mathcal{C}_{0}$), but $\mbox{diam}_{\mathcal{C}}(\Phi(x))$ is uniformly bounded, and what's more, there exists some $c=c(S)$ so that if $x,y$ are Riemann surfaces within $1$ of each other in $\mbox{Teich}(S)$, then 
\[ \mbox{diam}_{\mathcal{C}}(\Phi(x)\cup\Phi(y))\leq c. \]
A map $F$ from $\mbox{Teich}(S)$ to $\mathcal{C}(S)$ is then constructed by defining $F(x)$ to be  any element of $\Phi(x)$; an immediate consequence of the existence of $c$ is that $F$ is coarsely Lipschitz:  for any $x,y\in \mbox{Teich}(S)$, 
\[ d_{\mathcal{C}}(F(x),F(y))\leq c(s)\cdot d_{\mbox{Teich}}(x,y)+ c(s) \]

 In the final section, we show that $c$ can be taken to be independent of $S$:

\begin{theorem} There exists $c>0$ so that for any choice of $g,p$, 
\[ d_{\mathcal{C}(S_{g,p})}(F(x),F(y))\leq c\cdot d_{\mbox{Teich}(S_{g,p})}(x,y)+c \]
for any $x,y\in \mbox{Teich}(S_{g,p})$. In other words, the map sending a Riemann surface to any curve in $\mathcal{C}(S)$ with minimal extremal length is coarsely Lipschitz, with Lipschitz contants independent of the choice of $S$. 

\end{theorem}

 \textbf{How Uniform Hyperbolicity follows from Theorem $1.2$}

In both proofs of hyperbolicity of $\mathcal{C}$ by Masur-Minsky and Bowditch, the method of proof is to construct a family of quasigeodesics satisfying certain properties.

In \cite{Mas-Mur}, Masur and Minsky show the existence of a coarsely transitive family of quasigeodesics $\left\{g_{i}\right\}_{i\in I}$ (images of Teichm\"{u}ller geodesics under the above-mentioned map) equipped with projection maps $\phi_{i}:\mathcal{C}(S)\rightarrow g_{i}$, essentially having the property that the diameter of $\phi_{i}(\mathcal{C}(S)\setminus N_{\epsilon}(g_{i}))$ is bounded, with the bound depending only on $S$ and $\epsilon>0$. Here $N_{\epsilon}(\cdot)$ denotes the tubular neighborhood of radius $\epsilon$. It is then demonstrated that this is a sufficient condition for $\delta$-hyperbolicity. In \cite{Bow}, Bowditch constructs a similar family of uniform quasigeodesics and uses them to show that $\mathcal{C}(S)$ satisfies a subquadratic isoperimetric inequality, also implying $\delta$-hyperbolicity.

Each approach has its own implications; the Masur-Minsky program emphasizes greatly the connection between the geometry of the curve complex and that of Teichm\"{u}ller and hyperbolic space. Indeed, in order to prove the projection property mentioned above, they show a host of independently interesting results regarding the structure and combinatorics of train tracks on surfaces and Teichm\"{u}ller geometry. In contrast, Bowditch's approach yields more control on the actual size of the hyperbolicity constant.

 Specifically, both Bowditch and Masur-Minsky rely on a key lemma, that every unit area singular Euclidean surface homeomorphic to $S$ contains an annulus of definite width $W=W(S)$. Masur and Minsky prove this lemma using a limiting argument, while Bowditch's proof is more effective and yields some quantitative control on the growth of $W$ as a function of $S$. Uniform hyperbolicity of $\mathcal{C}_{1}$ follows if one plugs the result of Theorem $1.2$ into Bowditch's more effective set-up, as is demonstrated in section $4$.

We note that, while the conclusion of Theorem $1.2$ suffices to prove Theorem $1.1$, we conjecture that this lower bound can be improved: 
\begin{conjecture} Let $k,\alpha,\beta$ be as in the statement of Theorem $1.1$. Then there exists a polynomial $f_{k}:\mathbb{N}\rightarrow \mathbb{N}$ of degree $k-2$ such that
\[d_{\mathcal{C}}(\alpha,\beta)\geq k \Rightarrow i(\alpha,\beta)\geq f_{k}(\xi(S)).\]
\end{conjecture}

\textbf{Organization of the paper.} In section $2$, we establish Theorem $1.2$ for $k=3$ which is used in section $3$ as the base case of an induction argument on the curve complex distance. In section $3$, we complete the proof of Theorem $1.2$, and in section $4$ we show how Theorem $1.1$ immediately follows from Theorem $1.2$, together with the extensive quantitative control that Bowditch obtains on the growth of the hyperbolicity constant in his original proof. In section $5$, we use Theorem $1.2$ to prove Theorem $1.3$, and in section $6$, we show how to derive Theorem $1.4$ as a corollary of Theorem $1.2$. 

\textbf{Acknowledgements.} The author would primarily like to thank his adviser, Yair Minsky, for invaluable guidance and introducing him to the problem of obtaining quantitative control on the hyperbolicity of $\mathcal{C}$ as a function of complexity. The author would also like to thank Ian Biringer, Asaf Hadari, and Thomas Koberda for reading a draft of this paper and offering many helpful comments and suggestions. Finally, the author thanks Yael Algom-Kfir, Spencer Dowdall, Johanna Mangahas, and Babak Modami for many helpful and insightful conversations.

\section{Lower Bounds on Intersection Numbers for Filling Pairs}
Let $\Gamma=\left\{\gamma_{i}\right\}_{i\in I}$ be a collection of curves on $S_{g,p}$ in pairwise \textit{minimal position}, meaning that for each $k\neq j$, 
\[ |\gamma_{j}\cap \gamma_{k}|=i(\gamma_{j},\gamma_{k}). \]
We say that such a collection \textit{fills} the surface if the complement $S_{g,p}\setminus \Gamma$ is a union of topological disks and once-punctured disks. Equivalently, $\Gamma$ fills if every homotopically non-trivial, non-peripheral (i.e., not homotopic into a neighborhood of a puncture) simple isotopy class has a nonzero geometric intersection number with atleast one member of $\Gamma$. Henceforth, we will use the word \textit{essential} to refer to any curve which is non-peripheral and homotopically non-trivial. The study of Question 1 began with the following simpler question:

\begin{question} On $S_{g,p}$, how many times does a pair of simple closed curves need to intersect in order to fill? \end{question}

Note that two simple closed curves $\alpha$ and $\beta$ fill if and only if $d_{\mathcal{C}}(\alpha,\beta)\geq 3$, for this precisely means that there is no essential simple closed curve which is simultaneously disjoint from both $\alpha$ and $\beta$. In light of the notation used in the introduction, Question $3$ can therefore be restated as a question about $q_{3,g,p}$. 

\begin{lemma} $q_{3,g,p}\geq 2g+p-2$. \end{lemma}

\begin{proof}
Suppose $\alpha$ and $\beta$ fill $S_{g,p}$. Then $\alpha\cup\beta$ can be considered as the $1$-skeleton of a decomposition of $S_{g,p}$ into disks and once-punctured disks. Letting $D$ denote the number of disks in this decomposition, we obtain
\[\chi(S_{g,p})=2-2g-p=-i(\alpha,\beta)+D \]
The right hand side comes from the fact that there are twice as many edges as there are vertices in this decomposition, and the vertices are precisely the intersections between $\alpha$ and $\beta$. Then since $D\geq 0, q_{3,g,p}\geq  2g+p-2$. 
\end{proof}

\section{Proof of Theorem $1.2$}
In this section, we prove Theorem $1.2$: \vspace{2 mm}

\textbf{Theorem 1.2.} \textit{For each $\lambda\in (0,1)$, there is some $N=N(\lambda)\in \mathbb{N}$ such that if $\alpha,\beta \in \mathcal{C}_{0}(S_{g,p})$, whenever $\xi(S_{g,p})>N$ and  $d_{\mathcal{C}}(\alpha,\beta)\geq k$, 
\[ i(\alpha,\beta)\geq   \left(\frac{\xi(S)^{\lambda}}{f(\xi(S))}\right)^{k-2}  \]
where $f(\xi)=O(\log_{2}(\xi))$. }

\vspace{2 mm}

We will show that if $\xi>N(\lambda)$, then 
\[i(\alpha,\beta)<\left(\frac{\xi^{\lambda}}{f(\xi)}\right)^{k-2} \Rightarrow d_{\mathcal{C}}(\alpha,\beta)<k. \] 
 We induct on the curve complex distance $k$; the base case $k=3$ was established in section $2$. Both $N(\lambda)$ and $f(\xi)$ will be established in the course of the proof.

Thus we begin by assuming that $\alpha,\beta \in \mathcal{C}_{0}(S_{g,p})$ are such that 
\[ i(\alpha,\beta)<\left(\frac{\xi^{\lambda}}{f(\xi)}\right)^{k-2}. \]

Assume that $\alpha$ and $\beta$ are in minimal position.  Cutting along $\alpha$ yields $S':= S\setminus \alpha$, which is topologically either a genus $g-1$ surface with $p+2$ punctures, or $S'$ has two connected components. After cutting, $\beta$ becomes a set $B$ of disjoint arcs $\left\{b_{1},...,b_{n}\right\}$ with endpoints at the two punctures corresponding to $\alpha$.  Consider a maximal subcollection 
\[B'= \left\{b_{i_{1}},...,b_{i_{q}}\right\} \subseteq B\]
 of pairwise non-isotopic arcs; note that $B'$ must fill $S'$ and therefore there is some linear function $g(\xi)$ which bounds $q=|B'|$ from below. Furthermore, $q=|B'|\leq 6g+3(p+2)-4$ (see Lemma $2.1$ of \cite{Kor-Pap}). Choose $h(\xi)=O(\xi)$, and $h(\xi) \gg 6g+3(p+2)-4$, ($h(\xi)=50\xi$, for instance). Note that the number of complementary regions of the arcs in $B'$ is also no larger than $h(\xi)$. The bottom line is that $q=|B'|$ is bounded above and below by linear functions of $\xi$:

\[  g(\xi)\leq |B'|\leq 50\xi \] \vspace{3 mm}
\textbf{Case 1: The original surface $S$ is closed, so that $p=0$}:

 In this case, since $B'$ fills $S'$, the collection of complementary components of $B'$ in $S'$ consists of a disjoint union of polygons, each having at least $3$ sides. The collection of complementary regions defines a dual graph $G_{S'}=(V(G_{S'}),E(G_{S'}))$, whose vertices $V(G_{S'})$ correspond to complementary regions, and such that two vertices are connected by an edge if and ony if they represent adjacent (across an arc in $B'$) complementary regions in $S'$. For $v\in V(G_{S'})$, let $d(v)$ denote the degree of $v$, and let $\bar{d}(G)$ denote the average degree of $G$. Note that 
\[ \bar{d}(G_{S'})\geq 3. \]

 For each $j, 1\leq j\leq k$, define the mass of $b_{i_{j}}$, denoted $m(j)$, to be the number of arcs in the original collection $B$ that were isotopic to $b_{i_{j}}$. If $j$ is such that 
\[m(j) > \frac{i(\alpha,\beta)}{\xi^{\lambda}},\]
$b_{i_{j}}$ is said to have \textit{large mass}; otherwise $b_{i_{j}}$ has \textit{small mass}. Note that 

\[ \sum_{1\leq j\leq q}m(j)= i(\alpha,\beta), \]
and therefore there can be no more than $\xi^{\lambda}$ arcs of large mass. Assume $\xi$ is sufficiently large so that
\[ \xi \gg \xi^{\lambda}. \]
(This will be made more precise below.) Then for all such surfaces, cutting along all large mass arcs yields a possibly disconnected, non-simply connected (indeed, non-sporadic) surface $S''$. This is because cutting along any arc $b_{i_{j}}$ in $B'$ decreases the complexity by at most $3$; to see this, it suffices to analyze the three possible cases:
\begin{enumerate}
\item $b_{i_{j}}$ starts and ends at the same puncture and is non-separating;
\item $b_{i_{j}}$ starts and ends at the same puncture and is separating;
\item The terminal punctures of $b_{i_{j}}$ are distinct.
\end{enumerate}
The complexity of the surface obtained by cutting along $b_{i_{j}}$ depends only on which of the three cases we are in; the details of this are left to the reader, because the non-simply connectedness of $S''$ is implied by the conclusion of the following lemma:

\begin{proposition} There exists $N_{1}\in \mathbb{N}$ such that if $\xi>N_{1}$, there exists a homotopically non-trivial simple closed curve $\gamma \subset S''$ intersecting no more than $f(\xi)$ arcs of $B'$. 
\end{proposition}

In the statement of Proposition $3.1$, $f(\xi)$ is the same function $f$ from the statement of Theorem $1.2$, and it will be determined in the course of the proof. 
As will be proven below, proposition $3.1$ asserts that $\gamma$ is essential when viewed as a simple closed curve on the original surface $S$, not just as a curve on $S''$.

\textit{proof}: The arcs in $B''$ must fill $S''$ and therefore, as was the case with $B'$, these arcs cut the surface into finitely many simply connected regions. Denote by $G_{S''}$ the corresponding dual graph. Then 
\[ |V(G_{S''})|=|V(G_{S'})|,\]
and
\[|E(G_{S''})|\geq |E(G_{S'})|-\xi^{\lambda}\geq g(\xi)- \xi^{\lambda} \]
Note that some of the arcs in $B''$ may be isotopic on $S''$; indeed, if $T$ is a triangular complementary region in $S'$, and exactly one of the arcs of $\partial T$ has large mass, then the remaining two arcs will form a bigon in $S''$. It is also possible for all but one arc on the boundary of a complementary region in $S'$ to have large mass; if this happens, the remaining arc will bound a disk in $S''$. 

Choose $N_{1}$ large enough so that for all $\xi>N_{1}$,
\[ \bar{d}(G_{S''})\geq \frac{2(g(\xi)-\xi^{\lambda})}{|V(G_{S'})|}> 2.02 \]

We will need the following fact (see Lemma $3.2$ in \cite{Fio-Jor-The-Woo}; see also \cite{Die} for more details):

\begin{lemma} Let $\epsilon>0$. There exists a decreasing function $g:(0,1)\rightarrow \mathbb{R}_{+}$ so that if $G=(V,E)$ is any graph with $\bar{d}(G)>2+\epsilon$, then $G$ has girth no larger than $g(\epsilon)\cdot\log_{2}(|V|)$. 
\end{lemma}

We recall that the \textit{girth} of a graph $G$ is the length of the shortest cycle on $G$. 

 Thus assuming $\xi(S)>N_{1}$, by Lemma $3.2$, $G_{S''}$ has a simple cycle $\tilde{\gamma}$ no longer than 
\[ g\left(\frac{1}{50}\right)\cdot \log_{2}(2h(\xi)):= f(\xi) \]

(Recall that $h(\xi)=50\xi$, and was chosen to be linear in $\xi$, and to bound the number of non-isotopic arcs in the collection $B'$ from above.) 

By construction, $\tilde{\gamma}$ corresponds to a simple closed loop $\gamma$ intersecting no more than $f(\xi)$ small mass arcs; it remains to show that $\gamma$ is homotopically non-trivial in $S$. 

 By construction, $\gamma$ intersects at least one arc of $B''$, and if $\gamma$ intersects $b\in B''$, it only does so once since $\tilde{\gamma}$ is a simple cycle. We will show that all such intersections between $\gamma$ and $\beta$ are essential, which immediately implies that $\gamma$ is essential. 

Arguing by contradiction, assume that $\gamma$ and $\beta$ are not in minimal position. Then a closed arc $j_{\gamma}$ of $\gamma$ forms a bigon with a closed arc $j_{\beta}$ of $\beta$. Since $\gamma$ only intersects each arc in $B''$ once, it follows that $j_{\beta}$ must contain pieces of two distinct arcs $e_{x},e_{y}$ of $B''$ as sub-arcs, each of which contains one of the two points of the set $\left\{x,y\right\}= j_{\gamma}\cap j_{\beta} $.  Therefore, $j_{\beta}$ must contain an element of $\beta\cap\alpha$.

 Let $\tilde{\beta}$ represent the curve obtained from $\beta$ by replacing $j_{\beta}$ with $j_{\gamma}$. Then $\tilde{\beta}$ is isotopic to $\beta$ but intersects $\alpha$ fewer times than $\beta$ does, which contradicts the initial assumption that $\alpha$ and $\beta$ began in minimal position.

This concludes the proof of Proposition $3.1$. $\Box$ \vspace{2 mm}

Hence $\gamma$ is an essential, simple closed curve on $S$, which only intersects small mass arcs of $\beta$, and only at most $f(\xi)$ of them. Therefore,
\[ i(\gamma,\beta)\leq f(\xi)\cdot \frac{i(\alpha,\beta)}{\xi^{\lambda}}< \frac{f(\xi)}{\xi^{\lambda}}\cdot \left(\frac{\xi^{\lambda}}{f(\xi)}\right)^{k-2}= \left(\frac{\xi^{\lambda}}{f(\xi)}\right)^{k-3},  \]
where the strict inequality is due to our initial assumption. Then by the induction hypothesis, 
\[ d_{\mathcal{C}}(\gamma,\beta)<k-1,  \]
and thus by the triangle inequality, 
\[ d_{\mathcal{C}}(\alpha,\beta)\leq d_{\mathcal{C}}(\alpha,\gamma)+d_{\mathcal{C}}(\beta,\gamma)<1+(k-1)= k. \]

This concludes the proof of Theorem $1.2$ in the case that $S$ is closed.  \vspace{2 mm}

\textbf{Case 2: p>0}:

When $p>0$, the complementary regions of $B'$ in $S'$ may now be once-punctured, and could have a single arc of $B'$ in its boundary. In this case, we modify the definition of $G_{S'}=(V(G_{S'}),E(G_{S'}))$ as follows. There are two edges in the edge set $E(G_{S'})$ for every arc in $B'$. The vertex set $V(G_{S'})$ of $G_{S'}$ consists of two vertices for every non-punctured complementary region, and one vertex for each punctured complementary region. 

Note that this does not completely determine the graph $G_{S'}$. Given a non-punctured complementary region, there are choices to be made as to which edges of $G_{S'}$ connect to which of the two vertices corresponding to that region. However, the conclusion we are seeking will not depend on any of these choices, and therefore we assume they have been made arbitrarily. 

The abstract graph $G_{S'}$ comes equipped with a preferred immersion into $S'$, (whose image we also refer to as $G_{S'}$),  satisfying the following property: 

Let $e_{1},e_{2}$ be two edges corresponding to the same arc $e\in B'$ on the boundary of a punctured region $R$ in $S'$. Then the region bounded $e_{1}\cup e_{2}\cup e$ contains the puncture of $R$ (see Figure $1$).

 \begin{figure}
\centering
	\includegraphics[width=3.5in]{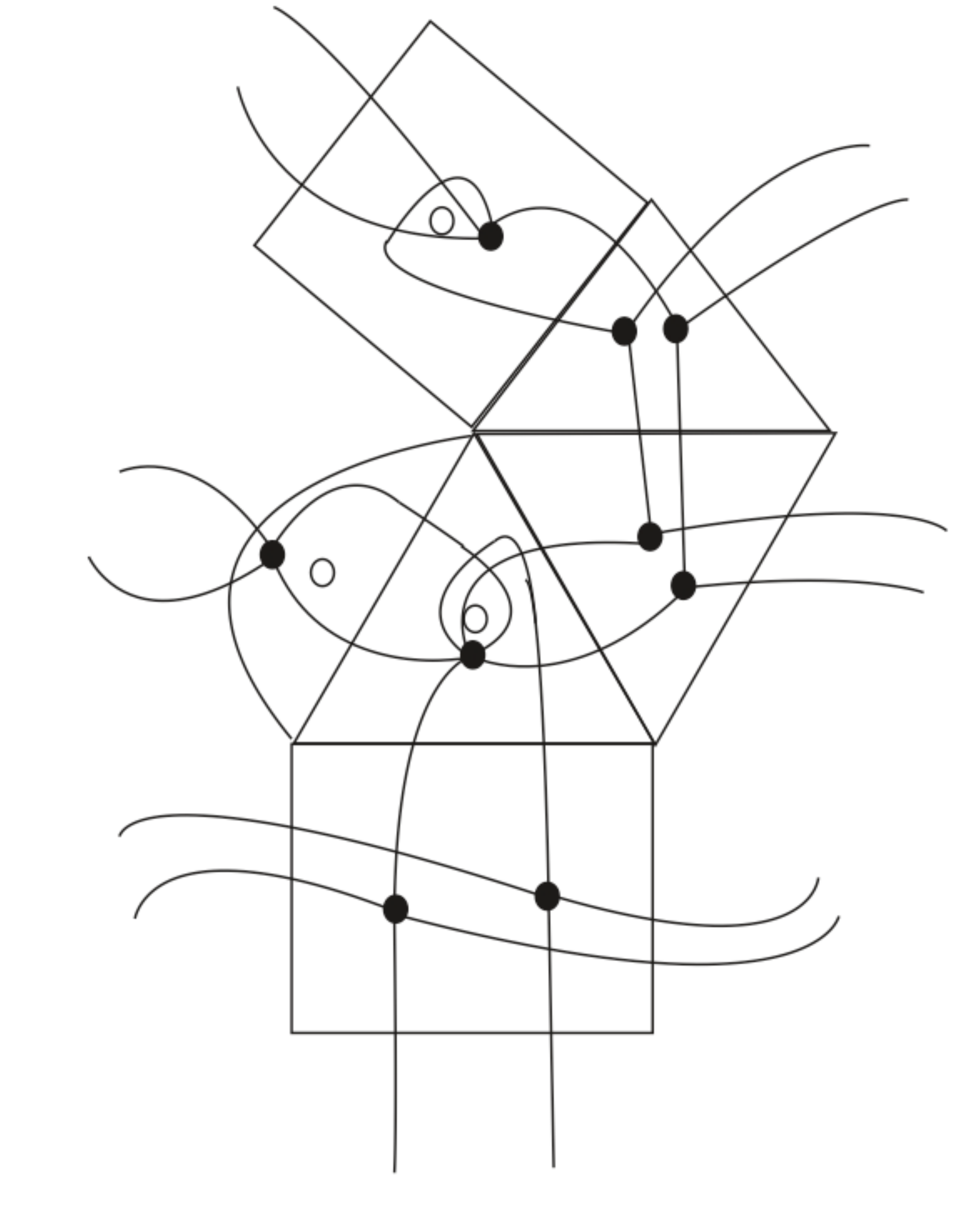}
\caption{ A possible portion of the graph $G_{S'}$ in the punctured case, immersed in $S'$. Punctures are represented by white circles, vertices of $G_{S'}$ by black circles. }
\end{figure}

The goal, as in the case that $p=0$, will be to show the existence of a simple closed curve $\gamma$, intersecting no more than $f(\xi)$ small mass arcs of $\beta$, counting multiplicity.

Denote by $P=\left\{R_{1},...,R_{s}\right\}$ the components of $S'\setminus B'$. As above, let $|R_{i}|$ denote the number of sides in the boundary of $R_{i}$. Let $U_{1}\subseteq P$ be the set of punctured regions, $U_{2}\subseteq P$ the set of non-punctured regions. Then by the Gauss-Bonnet theorem, 
\[ 2\pi|\chi(S')|= \pi \left(\sum_{R_{i}\in U_{1}}|R_{i}|+\sum_{R_{j}\in U_{2}}(|R_{j}|-2)\right) \]

\[\Rightarrow 2(2g+p-2)= \sum_{R_{i}\in U_{1}}|R_{i}|+\sum_{R_{j}\in U_{2}}(|R_{j}|-2) \]

\[\Rightarrow 2(2g+|U_{1}|-2)= \sum_{R_{i}\in U_{1}\cup U_{2}}|R_{i}|-2|U_{2}| \]

Note also that 

\[ |V(G_{S'})|=|U_{1}|+2|U_{2}| \hspace{1 mm} \mbox{and} \hspace{1 mm} |E(G_{S'})|= \sum_{R_{i}\in U_{1} \cup U_{2}}|R_{i}|.  \]

Since 
\[\bar{d}(G_{S'})= \frac{2|E(G_{S'})|}{|V(G_{S'})|}, \]

we have
\[ \bar{d}(G_{S'})= \frac{2\sum_{i}|R_{i}|}{|U_{1}|+2|U_{2}|}= \frac{2( 2(2g+|U_{1}|-2)+2|U_{2}|)}{|U_{1}|+2|U_{2}|} \]

\[= 2\cdot \frac{4g-4+2|U_{1}|+2|U_{2}|}{|U_{1}|+2|U_{2}|}\geq 2\cdot\frac{\xi(S)+p+2|U_{2}|}{p+2|U_{2}|}  \]

Note that the number $p+|U_{2}|$ of complementary regions $R_{i}$ is at most twice the number of the arcs in $B'$, since every such region is bounded by at least one arc, and each arc is adjacent to two regions. Therefore, 

\[ p+2|U_{2}|\leq 2(p+|U_{2}|)\leq 4(6g+3(p+2)-4)< h(\xi), \]

and hence 

\[ 2\cdot\frac{\xi(S)+p+2|U_{2}|}{p+2|U_{2}|} = 2 \left(1+ \frac{\xi}{p+2|U_{2}|} \right) \]
\[ > 2\left(1+\frac{\xi}{h(\xi)} \right) = 2.04 \]

 $G_{S''}$ is obtained from $G_{S'}$ as in the closed case; any edge in $E(G_{S'})$ corresponding to a large mass arc of $\beta$ is deleted, and $|V(G_{S''})|=|V(G_{S'})|$. 
Thus there exists some $N_{2}\in \mathbb{N}$ so that for $\xi(S)>N_{2}$, 

\[ \bar{d}(G_{S''})>2.02, \]
and therefore assuming $\xi>N_{2}$, $G_{S''}$ has a simple cycle $\tilde{\gamma}$ of length no more than 
\[ g\left(\frac{1}{50}\right)\log_{2}(|V(G_{S''})|) \leq f(\xi).  \]

Unlike in the closed case, $\tilde{\gamma}$ does not automatically correspond to a simple closed curve on $S''$, because $G_{S''}$ is only immersed and not embedded. Keeping this in mind, let $\gamma$ be a curve in the homotopy class of $\tilde{\gamma}$; we will first show that $\gamma$ is essential in $S$. We will again show this by proving that $\gamma$ is in minimal position with $\beta$; as above, assume that an arc $j_{\gamma}$ of $\gamma$ and an arc $j_{\beta}$ of $\beta$ bound a bigon in $S$. We claim that this bigon cannot be completely contained in $S''$, and therefore $j_{\beta}$ intersects $\alpha$. Then as in the closed case, homotoping $j_{\beta}$ across the bigon reduces the number of intersections between $\alpha$ and $\beta$, contradicting the assumption that $\alpha$ and $\beta$ were chosen to be in minimal position.  

Thus we must show that the bigon is not completely contained in $S''$. Note that, unlike in the closed case, it is now possible for $\gamma$ to cross a single arc of $\beta$ in $S''$ more than once. However, out of all of the arcs of $\beta$ entering the bigon, there must be some inner-most one, characterized by the property that together with a piece of $j_{\gamma}$, it bounds a bigon containing no arc of $\beta$ in its interior. This piece of $j_{\gamma}$ must then correspond to a segment of $\tilde{\gamma}$ contained in one complementary region of $S''$, and whose endpoints lie on the same arc on the boundary of this region. Thus this complementary region is punctured, and its puncture is contained in the interior of the bigon in question, a contradiction. Thus $\gamma$ is essential. 

Now, suppose $\gamma$ intersects itself once. Let $\gamma_{1}$ be one of the two simple closed curves obtained from $\gamma$ by starting and ending at the self-intersection point $x$.  If $\gamma_{1}$ is non-peripheral in $S$, replace $\gamma$ with $\gamma_{1}$.  Otherwise, let $\gamma_{2}$ be the other side of $\gamma$:
\[ \gamma_{2}=(\gamma\setminus \gamma_{1})\cup \left\{x\right\} \]

If $\gamma_{2}$ is non-peripheral in $S$, replace $\gamma$ with $\gamma_{2}$ and stop. We have reduced to the case where both $\gamma_{1}$ and $\gamma_{2}$ are peripheral, and therefore there exists punctures $y_{1}$, $y_{2}$ of $S$ so that $\gamma_{j}$ is homotopic into a neighborhood of $y_{j},j=1,2$. Note that since $\gamma$ is non-peripheral, $y_{1}\neq y_{2}$. Let $N(\gamma)$ be a small regular neighborhood of $\gamma$. Then the component $r$ of $\partial N(\gamma)$ encompassing both $y_{1}$ and $y_{2}$ is a simple curve, intersecting no more arcs of $\beta$ than does $\tilde{\gamma}$. It remains to show that $r$ is homotopically non-trivial and non-peripheral. 

 \begin{figure}
\centering
	\includegraphics[width=3.5in]{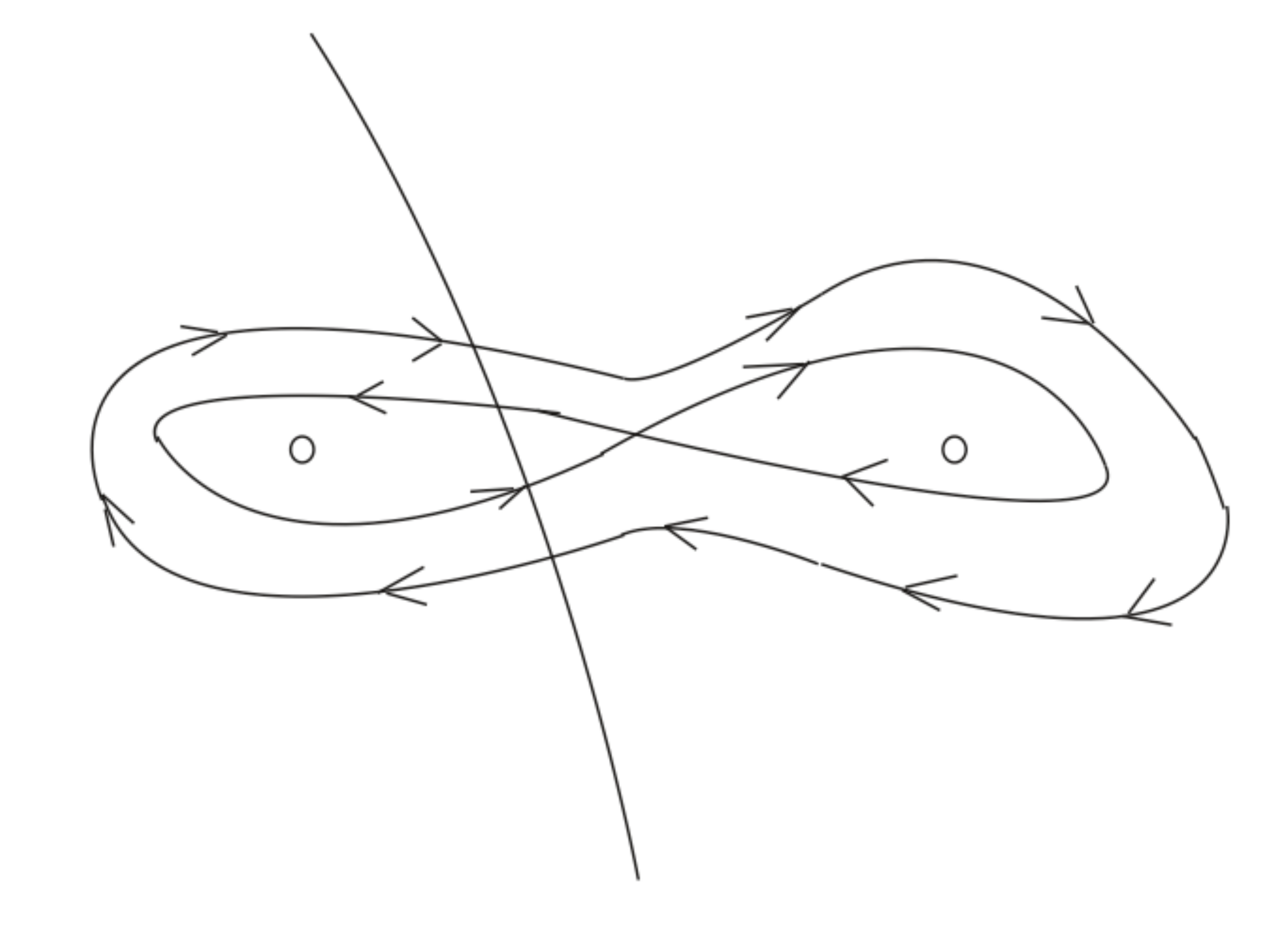}
\caption{ A schematic picture of $\gamma$ in the case of one self-intersection when $\gamma_{1},\gamma_{2}$ are both peripheral. White circles represent punctures, and $r$ is the simple curve encompassing $\gamma$. The transverse arc belongs to $\beta$, and it must intersect $r$ essentially. }
\end{figure}

To see that this is the case, note that there must be an arc of $\beta$ separating $y_{1}$ from $y_{2}$ that $\gamma$ crosses, and therefore $r$ crosses as well (see Figure $2$). By the same argument applied to $\gamma$ above, this intersection must be essential, and therefore $r$ is essential. Thus replace $\gamma$ with $r$.

Now suppose $\gamma$ has $k>1$ self-intersections, and let $\gamma_{1}$ be a simple closed curve obtained from $\gamma$ by starting and ending at some intersection point $x$ as above, and let $\gamma_{2}$ be the other side of $\gamma$. Then if $\gamma_{i}$ is non-peripheral for either $i=1,2$, replace $\gamma$ with $\gamma_{i}$. This reduces the number of self intersections that $\gamma$ possesses, so we are done by induction. If $\gamma_{1}$ and $\gamma_{2}$ are both peripheral, then we can homotope $\gamma$ into a neighborhood of the type pictured in Figure $2$, and as above, one of the boundary components of this neighborhood will be simple and essential, by the same argument used above.

Thus choose $N(\lambda)=\max\left\{N_{1},N_{2}\right\}$. This completes the proof of Theorem $1.2$. $\Box$

\section{Independence of the Hyperbolicity Constant on $\xi$}

In this section, we briefly summarize Bowditch's proof of hyperbolicity of $\mathcal{C}(S)$ as seen in \cite{Bow}, and then demonstrate how Theorem $1.1$ follows from his set-up, with the use of Theorem $1.2$. To avoid confusion, when possible we will use the same notation he introduces in his original article.

 Define $WX$ to be the set of \textit{weighted curves}; an element of $WX$ is a pair 

\[ (\lambda,\alpha); \alpha\in \mathcal{C}_{0}(S_{g,p}), \lambda\in \mathbb{R}_{+}.  \]

Then $WMX$, the set of $\textit{weighted multicurves}$ is the set of all finite formal sums of elements of $WX$ with pairwise disjoint summands. Note that any element of $\mathcal{C}_{0}(S_{g,p})$ is naturally considered as an element of either $WX$ or $WMX$ by assigning unit weight. Given $(p,\lambda_{1}),(q,\lambda_{2}) \in WX$, we can extend the notion of geometric intersection number for elements of $\mathcal{C}_{0}$ to $WX$ linearly
\[ i(\lambda_{1}p,\lambda_{2}q)=\lambda_{1}\lambda_{2}i(p,q), \]
and we can again extend linearly to $WMX$. Then for $\alpha,\beta,\delta \in WMX$ with $i(\alpha,\beta)=1$, define 

\[ l_{\alpha\beta}(\delta):= \max[i(\alpha,\delta),i(\beta,\delta)], \]
and
\[ m_{\alpha\beta}(\delta):= \sup\left[ \left\{ l_{\alpha\beta}(\delta) \right\} \cup \left\{\frac{i(\gamma,\delta)}{l(\gamma)} | \gamma \in \mathcal{C}_{0} \right\} \right]. \]
Then for any $r\geq 0$, we define 

\[ L(\alpha,\beta,r)=\left\{\delta \in \mathcal{C}_{0} | l_{\alpha\beta}(\delta)\leq r \right\}, \]
and 
\[ M(\alpha,\beta,r)=\left\{\delta\in \mathcal{C}_{0} | m_{\alpha\beta}(\delta)\leq r \right\}. \]

Note that $M(\alpha,\beta,r)\subseteq L(\alpha,\beta,r)$. In section $5$, Bowditch shows that there exists an essential annulus of width atleast $\eta=\eta(\xi)$ in any unit area singular Euclidean surface of complexity $\xi$, and uses this to show:

\begin{lemma} There exists $R=R(\xi)$ such that $M(\alpha,\beta,R)\neq \emptyset$. Furthermore, one can choose $R$ so that $R=O(\xi^{5/2})$. 
\end{lemma} 
 
As a consequence of this, we obtain 

\begin{lemma} There exists $D=D(\xi)$ so that for any $\alpha,\beta\in \mathcal{C}_{0}(S_{g,p})$ 
\[\mbox{diam}(L(\alpha,\beta,R)) <D \]
\end{lemma}

Indeed, Lemma $4.3$ follows immediately from the definitions and Lemma $4.2$: \vspace{1 mm}

There is some $\delta \in M(\alpha,\beta,R)$ since it is non-empty. Then for any $\gamma\in L(\alpha,\beta,R)$, we have 
\[ d(\gamma,\delta)\leq i(\gamma,\delta)+1\leq l_{\alpha\beta}(\gamma)m_{\alpha\beta}(\delta)+1\leq R^{2}+1. \]

Here, the first inequality is always true, independent of $\xi$ or the choice of $\gamma$ and $\delta$. The important inequalities in the chain above are the second and third ones, which in particular imply the existence of some $\delta\in L(\alpha,\beta, R)$ so that for any other $\gamma\in L(\alpha,\beta, R)$, $i(\gamma,\delta)\leq R^{2}+1$, and $R^{2}+1=O(\xi^{5})$. 

Hence, letting $f(\xi)$ be as in the statement of Theorem $1.2$,  there is some $\eta\in \mathbb{N}$ such that for $\xi>\eta$, 

\[ R^{2}(\xi)+1< \frac{\xi^{6}}{(f(\xi))^{7}} =\left(\frac{\xi^{\frac{6}{7}}}{f(\xi)}\right)^{(9-2)}\]

Therefore, for $N(6/7)$ the natural number in the statement of Theorem $1.2$, for all $\xi>\max(\eta,N(6/7))$, one has 
\[ d_{\mathcal{C}(S)}(\gamma,\delta) <9. \]

Thus, for all such values of $\xi$ we can use $D=18$ as the diameter bound for $L(\alpha,\beta,R)$.

 Bowditch considers a metric space $X$ having the property that to each pair of points $\alpha,\beta\in X$, there is a subset $\Lambda_{\alpha \beta}$, together with a \textit{coarse ordering} $\leq_{\alpha \beta}$ on $\Lambda_{\alpha \beta}$. By a coarse ordering, we mean a transitive relation satisfying the property that for any $x,y\in \Lambda_{\alpha \beta}$, either $x\leq_{\alpha \beta} y$ or $y\leq_{\alpha \beta} x$. Essentially, $(\Lambda_{\alpha \beta}, \leq_{\alpha \beta})$ is thought of as a coarsely parameterized line segment from $\alpha$ to $\beta$, where the parameterization is determined by the coarse ordering. 

Moreover, associated to $X$ is a function $\phi:X\times X\times X \rightarrow X$; given three points $x,y,z \in X$, $\phi(x,y,z)$ plays the role of the center of a triangle with vertices $x,y,z$. $\phi$ is required to satisfy the relations 
\[\phi(x,y,z)=\phi(y,z,x)=\phi(y,x,z), \hspace{1 mm} \mbox{and} \hspace{1 mm} \phi(x,x,y)=x, \]
 and define $\Lambda_{\alpha \beta}[xy]\subseteq \lambda_{\alpha \beta}$ by 
\[ \Lambda_{\alpha \beta}[xy] =\left\{z\in \Lambda_{\alpha \beta}: x\leq_{\alpha \beta} z \leq_{\alpha \beta} y \right\}. \]

He then shows: 

\begin{theorem} (\cite{Bow}) Suppose $(X, \left\{(\Lambda_{\alpha \beta})\right\}_{(\alpha,\beta)\in X\times X}, \phi)$ satisfies $\phi(\alpha,\beta,\gamma)\in \Lambda_{\alpha \beta} \cap \Lambda_{\beta\gamma} \cap \Lambda_{\gamma \alpha}$, and suppose there exists $K\geq 0$ satisfying 
\begin{enumerate}
\item $d_{\mathcal{H}}(\Lambda_{\alpha \beta}[\alpha,\phi(\alpha,\beta, \gamma)],\Lambda_{\alpha \gamma}[\alpha,\phi(\alpha,\beta,\gamma)])\leq K$,
\item given $x,y\in X$ with $d(x,y)\leq 1$, $\mbox{diam} (\lambda_{\alpha \beta}[\phi(\alpha,\beta,x),\phi(\alpha,\beta,y)])\leq K$, 
\item For $c\in \Lambda_{\alpha \beta}$, $\mbox{diam} (\Lambda_{\alpha \beta}[\gamma,\phi(\alpha,\beta,\gamma)]) \leq K$

\end{enumerate}
Then $X$ is  $\delta$-hyperbolic, with hyperbolicity constant depending only on $K$.

\end{theorem}

In Theorem $4.1$, $d_{\mathcal{H}}(\cdot,\cdot)$ denotes the Hausdorff distance.

 Bowditch then goes on to define the sets $\Lambda_{\alpha\beta}$ to essentially be the curve obtained by choosing an element of $L(\lambda\cdot \alpha,\mu\cdot \beta, R)$ for each pair $\lambda,\mu \in \mathbb{R}_{+}$ satisfying $i(\lambda\cdot \alpha, \mu\cdot \beta)=1$. Hyperbolicity is proved by showing that for this choice of $\Lambda_{\alpha\beta}$ (together with the choice of centers $\phi(\alpha,\beta,\gamma)$ whose definition is not summarized here- see \cite{Bow} ), the conditions of Theorem $4.1$ are satisfied, with $K$ depending only on $D$. This proves Theorem $1.1$. $\Box$

\section{Bounded Diameter of Vertex Cycle Sets of Birecurrent Train Tracks}
In this section, we prove Theorem $1.3$, but before doing so we recall some of the basic terminology of train tracks on surfaces (refer to \cite{Mas-Mur}, \cite{Pen-Har},  \cite{Pap} for more details). Recall that a \textit{train track} $\tau\subset S$ is an embedded $1$-complex; edges are called \textit{branches} and vertices \textit{switches}.  Each edge is a smooth parameterized path with well-defined tangent vectors at the endpoints. Furthermore,  at each switch $v$, there is a single line $L\subset T_{v}S$ such that the tangent vector of any branch incident at $v$ lies on $L$. For each switch $v$, we choose a preferred direction of $L$; a branch $b$ incident at $v$ is called \textit{incoming} if its tangent vector at $v$ is aligned with the chosen direction, and \textit{outgoing} otherwise. 
We require that the valence of each switch be at least $3$, unless $\tau$ has a simple closed curve component $c$; in this case we have a single bivalent switch on $c$. 

Any component $Q$ of $S\setminus \tau$  is a surface with boundary consisting of smooth arcs running between cusps. We define the $\textit{generalized Euler characteristic}$ of $Q$ to be 
\[ \chi(Q)-\frac{1}{2}V(Q)\]
where $V(Q)$ is the number of cusps on $\partial Q$. We require that the generalized Euler characteristic of each component of $S \setminus \tau$ be negative.

A \textit{train path} is a smooth sub-path of $\tau$ which traverses a switch only by entering via an incoming branch and exiting via an outgoing branch. Given a train track $\tau$, let $\mathcal{B}$ denote the set of branches. A non-negative, real-valued function $\mu: \mathcal{B}\rightarrow \mathbb{R}$ is called a $\textit{transverse measure}$ on $\tau$ if for each switch $v$, it satisfies
\[ \sum_{b\in i(v)}\mu(b)=\sum_{b'\in o(v)}\mu(b') \]
where $i(v)$ is the set of incoming branches at $v$, $o(v)$ the set of outgoing branches. These are called the \textit{switch conditions}. 

$\tau$ is called \textit{recurrent} if it admits a transverse measure with all positive weights, and \textit{transversely recurrent} if, for each branch $b$, there exists a simple closed curve $c=c(b)$ intersecting $b$, which intersects $\tau$ transversely and is such that $S\setminus(\tau \cup c)$ has no bigons. $\tau$ is called \textit{birecurrent} if it is both recurrent and transversely recurrent, and \textit{generic} if all switches are at most trivalent.

If $\mu$ is a transverse measure on $\tau$, then so is $\lambda\cdot \mu$ for any $\lambda \in \mathbb{R}_{+}$ because the switch conditions are linear. It follows that the set of all transverse measures, viewed as a subset of $\mathbb{R}^{|B|}$, is a cone over a compact polyhedron in projective space. Thus, it is often preferable to consider projective classes of transverse measures. Let $P(\tau)$ denote the projective polyhedron of transverse measures; the class $[\mu]\in P(\tau)$ is called a \textit{vertex cycle}  of $\tau$ if it is an extreme point of $P(\tau)$, that is to say that it cannot be written as a non-trivial convex combination of two other projective classes of measures in $P(\tau)$.

A geodesic lamination $\lambda$ is \textit{carried} by $\tau$ if there is a $C^{1}$ map $\phi:\lambda\rightarrow \tau$ which is isotopic to the identity, and such that the restriction of the differential $d\phi$ to any tangent line of $\lambda$ is non-singular; this amounts to saying that $F(\lambda)$ is a train path on $\tau$. Suppose $c$ is a simple closed curve carried by $\tau$. Then $c$ induces a transverse measure called the \textit{counting measure}: each branch $b$ of $\tau$ is assigned the integer corresponding to the number of times $c$ traverses $b$. 

In general, if $[v]$ is a vertex cycle of $\tau$, $[v]$ has a unique representative that is the counting measure on a simple closed curve $c$ carried by $\tau$ (see \cite{Mas-Mur}).  Thus if $[v_{1}],[v_{2}]$ are two vertex cycles of a train track $\tau$, we can define $d_{\mathcal{C}}([v_{1}],[v_{2}])$ to be the curve graph distance between their simple closed curve representatives.

In order to prove the results regarding nested train tracks needed to show hyperbolicity of $\mathcal{C}$ in \cite{Mas-Mur}, Masur and Minsky rely on the fact that there is a bound $B=B(S)$ such that any two vertex cycles of the same train track are a distance of at most $B$ from one another, which is an immediate consequence of the fact that there are only finitely many train tracks on $S$ up to homeomorphism. However they conjecture that $B$ can be made independent of $S$, and conjecture further that $B=3$ suffices.

Using Theorem $1.2$, we will show: \vspace{1 mm}

\textbf{Theorem 1.3}. \textit{Let $\tau\subset S_{g,p}$ be a train track, and let $\alpha,\beta$ be vertex cycles of $\tau$. Then if $C(S_{g,p})$ is sufficiently large, }
\[ d_{\mathcal{C}(S)}(\alpha,\beta)\leq 3. \] 
\vspace{1 mm}

We note that this constant $B$ also occurs in the proof of the quasiconvexity of the disk set by Masur and Minsky in \cite{Mas-MurIII}.

\textit{proof}: We can assume that $\tau$ is generic and birecurrent since for any train track $\tau$, there exists a generic birecurrent track $\tau'$ such that  $P(\tau)=P(\tau')$ (see \cite{Pen-Har}).  Assume $d_{\mathcal{C}}(\alpha,\beta)=4$, and assume further that $\xi(S_{g,p})>N(3/4)$. Then by Theorem $1.2$, 
\[ i(\alpha,\beta)\geq \left(\frac{\xi^{\frac{3}{4}}}{f(\xi)}\right)^{2}=\frac{\xi^{\frac{3}{2}}}{(f(\xi))^{2}} \]

We will need the following fact about vertex cycles (see \cite{Ham} for proof):

\begin{lemma} If $\alpha$ is a simple closed curve representative of a vertex cycle $[v]$ of $\tau$, then if $\phi:\alpha \rightarrow \tau$ is the associated carrying map, $\phi(\alpha)$ traverses each branch of $\tau$ at most twice, and never twice in the same direction. 
\end{lemma}

Since $\tau$ is birecurrent, for any $\epsilon>0$, there exists a hyperbolic metric $\sigma$ on $S$ so that $\tau$ has geodesic curvature less than $\epsilon$ with respect to $\sigma$ \cite{Pen-Har}.  Let $\alpha_{\sigma}$ (resp. $\beta_{\sigma}$) denote the unique geodesic representative of $\alpha$ (resp. $\beta$) in the metric $\sigma$. By choosing $\epsilon$ sufficiently small, we can assume $\alpha_{\sigma}$ and $\beta_{\sigma}$ both lie within a small embedded tubular neighborhood $N_{\epsilon'}(\tau)$ foliated by transverse ties. Collapsing these ties to points yields a train track isotopic to $\tau$ within some small bounded distance of $\tau$.

Let $\mbox{Br}(\tau)$ denote the number of branches of $\tau$. We show the following:
\begin{lemma} If $\alpha$ and $\beta$ are two vertex cycles of $\tau$, then 
\[ i(\alpha,\beta)\leq 4\mbox{Br}(\tau) \]
\end{lemma}
\textit{proof}: Let $b$ be any branch of $\tau$, and let $N_{\epsilon'}(b)$ denote the restriction of $\epsilon'$ tie-neighborhood of $\tau$ in the metric $\sigma$ to $b$. By Lemma $5.1$, $\alpha_{\sigma}$ and $\beta_{\sigma}$ can each enter $N_{\epsilon'}(b)$ at most twice. Since $\alpha_{\sigma}$ and $\beta_{\sigma}$ are both geodesics, they are in minimal position and therefore any arc of $\alpha_{\sigma}$ can intersect a given arc of $\beta_{\sigma}$ at most once within $N_{\epsilon'}(b)$. $\Box$

For any train track $\omega\subset S$, one has the bound (\cite{Pen-Har}) 

\[ \mbox{Br}(\omega)\leq -9\chi(S)-3p=18g+6p-18 \]

Thus there is some $k\in \mathbb{N}$ so that for all $\xi>k$, 
\[ \frac{\xi^{\frac{3}{2}}}{(f(\xi))^{2}} = i(\alpha,\beta) \gg 4\cdot \mbox{Br}(\tau),  \]
which contradicts Lemma $5.2$. Therefore for all surfaces $S$ satisfying $\xi(S)>\max(k,N(3/4))$, $\alpha$ and $\beta$ can not both be vertex cycles of the same train track on $S$, given that their curve graph distance is at least $4$.

 $\Box$

\section{Uniformity of the Lipschitz Constants for the Teichm\"{u}ller Projection Map}
In this final section, we prove Theorem $1.4$: \vspace{2 mm}

\textbf{Theorem 1.4}. \textit{There exists $c>0$ so that for any choice of $g,p$, 
\[ d_{\mathcal{C}(S_{g,p})}(F(x),F(y))\leq c\cdot d_{\mbox{Teich}(S_{g,p})}(x,y)+c \]
for any $x,y\in \mbox{Teich}(S_{g,p})$. In other words, the map sending a Riemann surface to any curve in $\mathcal{C}(S)$ with minimal extremal length is coarsely Lipschitz, with Lipschitz contants independent of the choice of $S$.  } \vspace{2 mm}

\textit{proof}:  We follow Masur-Minsky's proof of the complexity-dependent version of this statement, as seen in section $2$ of \cite{Mas-Mur}. Recall that 
\[\Phi:\mbox{Teich}(S)\rightarrow \mathcal{P}(\mathcal{C}_{0}(S))\]
 is the map that associates to each Riemann surface $x$ the set of isotopy classes of curves with smallest extremal length. By subdividing the Teichm\"{u}ller geodesic segment connecting $x$ to $y$ into unit length subsegments, It suffices to show that there exists some constant $c$ such that for any choice of $S$, given $x,y\in \mbox{Teich}(S)$ with $d_{\mbox{Teich}}(x,y)\leq 1$, 
\[ \mbox{diam}_{\mathcal{C}(S)}(\Phi(x)\cup\Phi(y))\leq c \]

For each complete, finite volume hyperbolic metric on $S$, there is an essential simple closed curve with hyperbolic length less than some constant $b=b(S)$, the \textit{Bers constant} ; it is known that $b(S)=O(\log(\xi(S))$ (see \cite{Bus}). 

What's more, the extremal length is bounded above by an exponential function of hyperbolic length (see \cite{Mas}): concretely, fixing a Riemann surface $x$ and an essential simple closed curve $\gamma$, let $\mbox{ext}_{x}(\gamma)$ denote extremal length, and $\mbox{hyp}_{x}(\gamma)$ the length of the geodesic representative of $\gamma$ in the unique complete finite volume hyperbolic metric in the conformal class of $x$. Then 

\[ \mbox{ext}_{x}(\gamma)\leq \frac{\mbox{hyp}_{x}(\gamma)}{2}e^{\mbox{hyp}_{x}(\gamma)/2} \]

Therefore, there is some constant $E=E(S)$ such that any Riemann surface in $\mbox{Teich}(S)$ has a curve with extremal length no more than $E(S)$, and $E(S)=O(\xi\log(\xi))$. 

Recall also Kerckhoff's characterization of $d_{\mbox{Teich}}$ in terms of extremal lengths (see \cite{Ker}):

\[  d_{\mbox{Teich}}(x,y)=\frac{1}{2}\sup_{\gamma\in \mathcal{C}_{0}(S)}\frac{\mbox{ext}_{y}(\gamma)}{\mbox{ext}_{x}(\gamma)} \]

Now, let $\alpha\in \Phi(x)$; since $\alpha$ minimizes extremal length in $x$, $\mbox{ext}_{x}(\alpha)\leq E(S)$. Since $d_{\mbox{Teich}}(x,y)\leq 1$, by Kerckhoff's formula for the Teichm\"{u}ller distance, given any $\beta\in \Phi(y)$, it follows that $\mbox{ext}_{x}(\beta)\leq e^{2}E(S)$. As seen in both \cite{Mas-Mur} and \cite{Min}, 
\[ \mbox{ext}_{x}(\alpha)\mbox{ext}_{x}(\beta)\geq i(\alpha,\beta)^{2}\]
and therefore 
\[ i(\alpha,\beta)\leq eE(S)=o(\xi(S)^{2}) \]
Thus Theorem $1.2$ implies that there is some $k\in \mathbb{N}$ so that for $\xi(S)>k$, 
\[ d_{\mathcal{C}(S)}(\alpha,\beta)\leq 4. \]

Independent of the choice of complexity, curve complex distance is bounded above by a logarithmic function of intersection number (see \cite{Bow}):
\[ d_{\mathcal{C}}(\alpha,\beta)\leq f(i(\alpha,\beta)), f:\mathbb{N}\rightarrow \mathbb{N} \hspace{1 mm} \mbox{and} \hspace{1 mm} f=O(\log(n)) \]
Therefore it suffices to choose
\[ c=\max(4, \max_{ \xi(S)<k}f(\ceil*{eE(S)}) )    \]

 Here, $\ceil*{x}$ denotes the smallest integer larger than $x$. This completes the proof of Theorem $1.5$.  $\Box$

\end{document}